\documentclass[10pt]{amsart}
\usepackage[margin=1in]{geometry}
\input{htpy-unitary.sty}
\makeatletter
\renewcommand{\email}[2][]{%
  \ifx\emails\@empty\relax\else{\g@addto@macro\emails{,\space}}\fi%
  \@ifnotempty{#1}{\g@addto@macro\emails{\textrm{(#1)}\space}}%
  \g@addto@macro\emails{#2}%
}
\makeatother

\title[High-corank torsion in homotopy of unitary groups]
{High-corank torsion in homotopy of unitary groups via topological modular forms and higher real $K$-theories}
\author{Hood Chatham}
\author{Yang Hu}
\author{Morgan Opie}
\address[H. Chatham]{Cloudflare}
\email{roberthoodchatham@gmail.com}
\address[Y. Hu]{University of Regina}
\email{yang.hu@regina.ca}
\address[M. Opie]{Northwestern University}
\email{mpopie@northwestern.edu}

\date{}              

\begin{document}

\begin{abstract}
Work of Toda identifies groups of metastable vector bundles on even-dimensional spheres with stable homotopy groups of certain stunted projective spectra. Using Weiss' unitary calculus, the second-named author generalized this identification to show that metastable, stably trivial vector bundles on even cell complexes can naturally be identified with stable homotopy classes of maps into a shifted stunted projective spectrum. Thus, certain classical questions about vector bundles (or homotopy of unitary groups) can be rephrased as stable computations. In this note, we show that certain generalized cohomology theories arising in chromatic and equivariant homotopy theory can be used to deduce the existence of non-trivial, stably trivial vector bundles on spheres and complex projective spaces.
\end{abstract}
\maketitle
\tableofcontents


\section{Introduction}

The homotopy groups of unitary groups are a classical subject that saw great progress from the late 1950s 
through the early 1980s, with notable contributions by Bott, Borel, Hirzebruch, Imanashi, Kervaire, Matsunaga, Mimura, Mosher, Nishida, Toda, and others.
The most famous of these results is Bott's theorem \cite[Equation (1.5)]{Bott}, stating that for $r>n\geq 0$,
\[\pi_{2n+1}U(r) \cong \Z, \,\, \pi_{2n}U(r)=0.\]
This isomorphism arises from Bott's periodicity theorem for the homotopy of the stable unitary group $U$; if $k\geq 2r$, the group $\pi_{k}U(r)$ is unstable, meaning that it cannot be inferred from those of $U$. For $n\geq 2$, Borel and Hirzebruch \cite[26.10]{BorelHirz} computed the first unstable group:
\[ \pi_{2r}U(r) \cong \Z/(r!).\]
The next group $\pi_{2r+1}U(r)$ is either $0$ if $r$ is odd, or $\Z/2$ if $r$ is even \cite[Lemma I.4]{kervaire1960}. Continuing, there are many formulas for groups of the form $\pi_{2r+j}U(r)$ for $j$ small; we refer the reader to \cite{Mimura_HBAT} for an extensive summary, and to the introduction of \cite{Oshima80} for a brief literature review.  Given a fixed odd $j\leq 5$ and $r$ large enough, one observes from \cite[page 971]{Mimura_HBAT} that the groups $\pi_{2r+j}U(r)$ depend only on $r \pmod{24}$. It turns out that there is a good reason for this, and, indeed, there is an explanation for certain patterns in the unstable homotopy groups of unitary groups in the entire {\em metastable range}.
\begin{defn}\label{def:metastable}
A homotopy group $\pi_{2r+j}U(r)$ is metastable if $0\leq j\leq 2r-1$. More generally, if $X$ is a finite cell complex with cells in dimensions at most $k=2r+j$, we say that complex rank $r$ topological bundles on $X$ are metastable if $1\leq j \leq 2r$. \end{defn}
\begin{rmk} Informally, the metastable range consists of non-stable bundles (rank less than dimension) that have real rank at least half the dimension of the cell complex.  Note that complex topological rank $r$ vector bundles correspond to homotopy classes of maps to the classifying space $BU(r)$, and hence the two versions of \Cref{def:metastable} coincide for vector bundles on spheres. \end{rmk}

 For odd $j$ in the metastable range, the groups $\pi_{2r+j}U(r)$ are closely related to {\em stable} homotopy groups of stunted projective spectra  \cite[Theorem 4.3]{Toda59}. Classically, this observation comes from a direct analysis of the cellular structure of $SU(r)$. For $j$ odd and metastable, the identification is particularly simple: 
\begin{equation}\label[empty]{eq:toda}\pi_{2r+j}U(r) \cong \pi_{2r+j}^{\op{s}} \cp{r}{n},\end{equation}
where $n=\frac{2r+j+1}{2}$, $\cp{r}{n}$ is the cofiber of the cellular inclusion $\cp{}{r-1}\to \cp{}{n}$, and $\pi_*^{\op{s}}$ denotes stable homotopy. When $j$ is even and in the metastable range, the left-hand side of \Cref{eq:toda} is replaced by a quotient of the stable homotopy of a stunted projective spectrum by a particular cyclic subgroup. In particular, this explains the fact that odd homotopy groups of unitary groups $\pi_{2r+j}U(r)$ have formulas in terms of $r$ modulo finitely many primes: for $n-r$ fixed, the stable cell structure of $\cp{r}{n}$ depends only on $r$ modulo the James number of $\cp{}{n-r}$ (see \Cref{rmk:James} below for discussion).

\begin{rmk} Recall that $\cp{r}{n}$ admits a stable, $p$-complete Adams decomposition as a direct sum of $p-1$ spectra; the cells in the $i$-th summand are precisely those in degrees congruent to $2i \pmod{2p-2}$. This structure matches up with the $p$-complete unstable splitting of $SU(r)$ provided by 
Mimura--Nishida--Toda \cite[Section 3]{mimura77}. \end{rmk}

It is not an exaggeration to say that the spectra $\cp{r}{n}$ are some of the best-understood objects in stable homotopy theory. Indeed, the stable cell structure of each $\cp{r}{n}$ is determined by the stable cell structure of $\cp{}{\infty}$. 
For any given $n$ and $r$, machines of computational stable homotopy theory provide many methods to compute $\pi_{*}^{\op{s}}\cp{r}{n}$. As such, the stable homotopy groups of $\cp{r}{n}$ for fixed $n-r$ small have been extensively and systematically studied. 

The homotopy groups $\pi_{2r+j}\cp{r}{n}$ for $2n\geq 2r+j$, $r,n$ metastable, 
and $j\leq 9$ are completely known. Implications for unstable homotopy groups of unitary groups are written down in a number of places. The unstable group $\pi_{2r+2}U(r)$ was computed by Toda \cite[Theorem 4.4]{Toda59} and independently by Kervaire \cite[page 161]{kervaire1960}. The groups $\pi_{2r+j}U(r)$ for $j=3,4,5$ are found in \cite[Theorems I, II \& Section 9]{matsunaga61} (with corrections found in \cite{matsunaga62correction}). The $j=6$ case was resolved in \cite[Theorem 1]{matsunaga63}, while the $j=7$ case was resolved by \cite[Theorem 2]{matsunaga63} for $2$-power torsion, in combination with \cite[Theorem 1]{imanishi67} for odd-prime-power torsion. \cite[Theorem 1]{imanishi67} and \cite[Main Theorem]{Matsunaga64} apply much more generally, dealing with odd-primary torsion for a range of $j$. In particular, these cases resolve the odd-primary component of $\pi_{2r+8}U(r)$ and $\pi_{2r+9}U(r)$; the full groups are computed by combining these works with $2$-primary computations from \cite[Theorem 2.1]{mosherstunted} for $j=8$, and both \cite[Theorem 2.1]{mosherstunted} and \cite[Main Theorem]{Oshima80} for $j=9$.

For $j\geq 10$, there are many partial results on $\pi_{2r+j}U(r)$, focusing on $p$-power torsion for a fixed prime that is not too small relative to $j$ (e.g., see again \cite[Theorem 2.1]{mosherstunted} at $p=2$, and \cite[Theorem 1]{imanishi67} and \cite[Main Theorem]{Matsunaga64} at odd primes). However, there is a great deal that is {\em not} known for $j$ large relative to the prime.

The primary goal of this note is to use the identification \Cref{eq:toda} to detect $p$-power torsion in the unstable homotopy groups $\pi_{2r+j}U(r)$, for $j$ metastable and larger than the values classically studied. We rely on a simple idea in homotopy theory: rather than compute homotopy of a given spectrum, we might instead compute $R$-homology for $R$ some easier generalized cohomology theory. Given a good understanding of $R$-homology, we can hope to work backwards to deduce facts about homotopy. 
\begin{rmk}\label{rmk:feasible} This method can feasibly give information beyond the classical ranges due to relatively recent computations of the homotopy groups $\pi_*R$ and Hurewicz images $\op{Im}(\pi_*\sphere \to \pi_*R)$ for various ``designer cohomology theories,'' e.g., {\em topological modular forms} or {\em higher real $K$-theories}. We discuss these theories and some facts about them in \Cref{subsec:generalized-coh}. \end{rmk}
Applying this with various choices of $R$, we detect families of non-trivial $p$-power torsion in $\pi_*^{\op{s}}\cp{r}{n}$, and deduce:

\begin{thmx}[See 
\Cref{cor:KO_htpy_BU}, \Cref{cor:tmf2_htpy_BU}, \Cref{cor:bel-shim2} and \Cref{example:general_odd_prime2}]\label{thm:D}
\hspace{.5in}
\begin{itemize}
\item[i.] For each $t\geq 0$ and each $i \geq 2t+1$, there is nonzero $2$-torsion in $\pi_{4(2t+1+i)-3}U(2i).$

\item[ii.] For each $t \geq 0$,
$i\geq 12t+3$, and $j \geq 12t+8$, let $m(i)= 8(12t+3+i)-1$ and $n(j)=8(12t+8+j)-1$. There is nonzero $2$-torsion in $\pi_{2m(i)-1}U(8i)$ and $\pi_{2n(j)-1}U(8j)$.
\item[iii.]  For each $t\geq 0$ and $l \geq 12t+7$, there is nonzero $3$-torsion in
    $\pi_{2\left(3l+19+36t\right)-1} U(3l)$.
\item[iv.]  Let $p\geq 5$ be prime. Let $d_j=2p(p-1)^2+2p-3+j(2p^2-2p-2)$. For each $1\leq j \leq p-2$ and each $l\geq
    \frac{d_j+1}{2p}$, there is nonzero $p$-torsion in $\pi_{2lp+1+d_j-1}U(lp)$.
\end{itemize}
\end{thmx}
\begin{rmk} In (i) above, $2$-torsion is detected in $KO$-homology, where $KO$ denotes the real $K$-theory spectrum. In (ii), $2$-torsion is detected using the spectrum of topological modular forms localized at the prime $2$. In (iii)-(iv), $p$-torsion is detected using higher real $K$-theories of height $p-1$ at a prime $p$. 
\end{rmk}

Since we know so much about complex topological vector bundles on spheres in the metastable range, it is natural to ask if these results can be propagated to other spaces. Recall that this works out beautifully in the stable range: Bott periodicity is not only a statement about homotopy groups of unitary groups, but also a statement relating the complex $K$-theory of any space to the complex $K$-theory of any complex Thom space over it. 

 In the metastable range, this idea can be effectively carried out for complex projective spaces, using work of the second-named author. Let $\Vect_r^0(\cp{}{n})$ denote the set of isomorphism classes of stably trivial rank $r$ complex topological vector bundles on the complex projective space $\cp{}{n}$. For $n$ and $r$ metastable, which translates to $\frac{n}{2} \leq r <n$, \cite[Theorem 2.1]{Hu} gives an identification

\begin{equation}\label[empty]{thm:yang} \Vect_r^0(\cp{}{n})\cong \{\cp{r}{n},\Sigma \cp{r}{n}\},\end{equation}
where $ \{\cp{r}{n},\Sigma \cp{r}{n}\}$ denotes stable homotopy classes of maps between the indicated finite cell complexes. 
\begin{rmk} 
As discussed in \Cref{cor:iso} below, the identification \Cref{thm:yang} can be generalized. Viewed appropriately, it is moreover a direct generalization of Toda's theorem \Cref{eq:toda}. \end{rmk}

Now consider the following geometric construction. Given an interesting new complex topological vector bundle on an even sphere $S^{2n}$, we can study its pullback under the natural map $\cp{}{n} \to S^{2n}$ given by projection onto the top cell. Using \Cref{thm:yang}, we show that, in many cases, the bundles from \Cref{thm:D} pull back to non-trivial, stably trivial vector bundles on the appropriate complex projective space.

\begin{thmx}[\Cref{cor:KO_ah}, \Cref{cor:tmf_ah},  \Cref{example:bel-shim}, and \Cref{example:general_odd_prime}]\label{thm:cpns} Let $V \in \pi_{2n}BU(r)$ be a non-zero $p$-torsion bundle on $S^{2n}$ corresponding to a bundle from parts (i)-(iii), or (iv) with $j=1$, of \Cref{thm:D}. Let $f\: \cp{}{n} \to  S^{2n}$ denote projection onto the top cell, i.e., the cofiber of the cellular inclusion $\cp{}{n-1} \to \cp{}{n}$. Then $f^*V$ is a non-trivial, stably trivial vector bundle on $\cp{}{n}$.
 \end{thmx}

\subsection{Conventions}\label{subsec:conventions}
\begin{itemize}
\item Given spaces $X$ and $Y$, we write $[X,Y]$ for homotopy classes of pointed maps from $X$ to $Y$.
\item Given spectra $X$ and $Y$, write $\hmapsp{X}{Y}$ for homotopy classes of maps of spectra from $X$ to $Y$. 
\item We will frequently consider Morava $E$-theory, and higher real $K$-theory spectra, of height $p-1$ at a prime $p$. These are usually abbreviated by $\e_{p\,\text{-}1}$ and $\eop$, respectively. When there is no potential for ambiguity, the spectrum $\eop$ will be abbreviated by $\eo$ and $\e_{p\,\text{-}1}$ by $\e$.
\item Given $\eo$-modules $A$ and $B$, we write $\eomod(A,B)$ for the space of $\eo$-module maps from $A$ to $B$.  
\item We write $\otimes$ for the smash product of spectra.
\item Given a finite spectrum $X$, we write $DX$ for the Spanier--Whitehead dual of $X$.
\item We write $H^*$ for $H^*(-;\fp)$ and $H_*$ for $H_*(-;\fp)$. \item For $p$ a prime, let $P(1)^*$ be the sub-Hopf algebra of the mod  $p$-Steenrod algebra generated by the power operation  $P^1$. Let $P(1)_*\simeq \F_p[\xi_1]/\xi_1^p$ denote the  dual quotient Hopf algebra of the dual Steenrod algebra.  \item  Let $W_l$ denote the unique indecomposable length  $l$ graded $P(1)_*$-comodule, concentrated in degrees $0$ through $2(l-1)(p-1)$. \item Let $\Sigma^sW_l$ denote the indecomposable length $l$ graded $P(1)_*$-comodule concentrated in degrees $s$ through $s+2(l-1)(p-1)$. \item  Let $X_l$ denote the spectrum uniquely determined by the property that $H_*(X_l;\Fp) \simeq W_l$ 
as a $P(1)_*$-comodule and such that all Steenrod operations $P^i$ for $i>1$ act by zero. \item We write $\gamma_1$ for the canonical bundle on $\cp{}{n}$ (possibly with $n=\infty$) determined up to isomorphism by requiring its first Chern class to be one. This is the dual of the tautological bundle.
\end{itemize}

\subsection{Acknowledgements}
The authors would like to thank Dev Sinha for his encouragement of this project, including travel funding provided under the Simons 
Foundation Collaboration Grant 422618. The authors also benefited greatly from
the opportunity to collaborate during South Central Topology Conference III.
While working on this project, H.C. and M.O. were supported by the National Science Foundation under Award
Numbers 2002087 and 2202914, respectively.

\section{Background}\label{hist}
In this section, we survey some relevant background. This includes a discussion of Weiss' unitary calculus and its applications to metastable, stably trivial complex topological vector bundles (\Cref{subsec:weiss}), and a survey of generalized cohomology theories that are useful for detecting nontrivial, high-corank vector bundles on both spheres and complex projective spaces (\Cref{subsec:generalized-coh}).

\subsection{Unitary calculus and metastable, stably trivial complex topological vector bundles}\label{subsec:weiss}

Weiss' unitary calculus is a method for studying continuous functors from a category of complex,
finite-dimensional inner product spaces and their linear isometric inclusions to the
category of pointed topological spaces \cite{Weiss}. For such a functor $F$, the unitary calculus machine produces a tower of fibrations
\begin{equation}\label[diagram]{tower}\cdots \longrightarrow T_nF \longrightarrow T_{n-1}F 
\longrightarrow \cdots \longrightarrow T_1F \longrightarrow T_0F\end{equation}
known as the {\em Weiss tower} or {\em unitary tower} of $F$. The functors $T_nF$ should be thought of as approximations of $F$ by nicer functors. More precisely, for each
$n\geq 0$, $T_nF$ is an {\em $n$-polynomial} functor in the sense of Weiss \cite[Definition 5.1]{Weiss},
and there is a canonical natural transformation $F \to T_nF$.

\Cref{tower} is said to {\em converge} to $F$ if $F \rightarrow \holim_n
T_nF$ is a pointwise weak equivalence.
We define $L_nF$ to be the homotopy fiber (in the functor category) of the map $T_nF \rightarrow T_{n-1}F$, known as the {\em $n$-th layer} of $F$. $L_nF$ is an
{\em $n$-homogeneous} functor. In \cite[Theorem 7.3]{Weiss}, Weiss classifies $n$-homogeneous 
functors as being determined by certain naive $U(n)$-spectra. While this abstract result is hard to understand in full generality, it turns out to be surprisingly simple for $n=0$ and $n=1$, especially for specific examples. For $n=0$, the functor $T_0F$ of the Weiss tower is
\[(T_0F)(V) = \hocolim_k F(V\oplus \bb{C}^k),\] 
which is a constant functor.

Consider the functor $F=BU(-)$, given on objects by $V \mapsto BU(V)$. 
In this case, $T_0\left( BU(-) \right)$ is the constant 
functor $V\mapsto BU$ whose target classifies stable bundles. 
The layers $L_i(BU(-))$ have been studied by 
Arone \cite[Theorem 2 and Theorem 4]{Arone02}. For example, Arone's description implies that, for $r$ the complex dimension of 
$V$ and $n\geq 2$, the space $L_nBU(V)$ is at least $4r$-connected. Moreover, if $n=1$, Arone's description of the functor $L_1BU(V)$ is actually very explicit: by \cite[Theorem 2.1]{Hu},  up to homotopy, we can identify
\[ L_1BU(V) \cong \Omega^\infty \Sigma \cp{r}{\infty},\]
where $r=\dim_{\C}(V)$.

Moreover, Arone's connectivity estimate, together with \cite[Theorem 2.1]{Hu} and the fiber sequence $L_1BU(V) \to T_1BU(V) \to BU$, implies the following:
\begin{cor}\label{cor:iso} Let $X$ be a finite cell complex with cells in dimensions at most $4r$. Suppose also that $X$ is cohomologically even, or that $[X,U]=0$. Then stably trivial rank $r$ complex topological vector bundles on $X$ up to isomorphism are in bijection with (stable) homotopy classes of maps from $\Sigma^\infty X$ to the stunted projective spectrum $\Sigma \cp{r}{\infty}$. \end{cor}

Applying this result to $X=S^{2n}$ for $\frac{n}{2}\leq r <n$ recovers Toda's result: such homotopy groups of unitary groups are identified with the (stable) homotopy groups
\begin{equation}\label[empty]{groupA} \pi_{2n}\Sigma \cp{r}{n},\end{equation} where the spectrum $\cp{r}{n}$  is the cofiber of the natural cellular inclusion $ \Sigma^\infty\cp{}{r-1} \to  \Sigma^\infty \cp{}{n}$. 

Applying \Cref{cor:iso} to $X=\cp{}{n}$ for $\frac{n}{2}\leq r <n$ implies 
\begin{equation}\label[empty]{groupB}\Vect_r^0(\cp{}{n}) \cong \hmapsp{ \cp{}{n}}{\Sigma \cp{r}{\infty}}\cong \hmapsp{ \cp{r}{n}}{ \Sigma \cp{r}{n}} ,\end{equation}
where $\Vect_r^0(\cp{}{n})$ denotes isomorphism classes of stably trivial rank $r$ complex topological bundles on $\cp{}{n}$, and $\hmapsp{-}{-}$ denotes homotopy classes of maps of spectra.

We can consider the groups \Cref{groupA} and \Cref{groupB} as $n$ and $r$ both vary within the metastable range $\frac{n}{2}\leq r <n$. To see patterns, it turns out to be convenient to fix $c=n-r$, corresponding to a fixed corank vector bundle. Once $c$ is fixed, James periodicity implies that both groups \Cref{groupA} and \Cref{groupB} are periodic in $r$ with period the James number $M_{c+1}$ of the tautological bundle on $\cp{}{c}$ \cite[Theorems 2.3 and 2.7]{JamesStiefel}. 

\begin{rmk}\label{rmk:James} Explicitly, the James number is the smallest integer $k$ so that the $k$-fold Whitney sum of the tautological bundle with itself is spherically trivial. Using that $\cp{r}{r+c}$ is the Thom space of $r$ times the tautological bundle on $\cp{}{c}$, we see that
\[ \cp{r+M_{c+1}}{r+c+M_{c+1}} \simeq \Sigma^{2M_{c+1}} \cp{r}{r+c},\]
which gives the claimed periodicity for both groups. For more discussion, see \cite[Remark 1.2]{HMY}; for explicit formulas, see \cite[Page 100]{AdamsWalker}.
\end{rmk}

For any given $n$ and $r$, one can appeal to standard stable computational tools (e.g., Adams or Atiyah--Hirzebruch spectral sequences) to compute \Cref{groupA} or \Cref{groupB}. Combining this with the periodicity observed above, calculating all corank $c$ metastable, stably trivial complex topological vector bundles either on even-dimensional spheres or on complex projective spaces reduces to exactly $M_{c+1}$ distinct computations. For $c \leq 3$, the full computations for spheres can be read off from \cite[page 971]{Mimura_HBAT}; those for complex projective spaces are summarized in \cite[Theorems 1.1 and 1.2]{Hu} and \cite[Theorem A]{HMY}. For $c\geq 4$, the relevant periodicity is at least $2880$. A complete computation becomes hard to organize.

Rather than attempt a full computation for fixed corank, we might instead try to detect certain interesting patterns in the structure of the spectra $\cp{r}{n}$ by probing these spaces and the mapping spectra between them by generalized cohomology.  For example, we work at a specific prime and chromatic height to isolate certain interesting stably trivial, non-trivial vector bundles. This is the approach taken in both \cite{HMY} and this paper. 

\subsection{
   Generalized cohomology theories with useful Hurewicz images
}\label{subsec:generalized-coh} 

Let $R$ be a generalized cohomology theory, and $X$ an even-cell complex with cells in degree at most $2n$, where $n\leq2r$. By \Cref{cor:iso}, we have a bijection

\[ \Vect_r^0(X) \cong \hmapsp{\Sigma^\infty{X}}{\Sigma \cp{r}{n}}.\] 

The left-hand side is the zeroth homotopy group of a mapping space in the stable category of spectra, or $\sphere$-modules, where $\sphere$ denotes the sphere spectrum. Suppose that we have some other $E_\infty$-ring spectrum $R$; that we understand $R$ better than $\sphere$; and have some handle on the {\em Hurewicz image} for $R$, i.e., the image of the map induced by the unit map $\sphere \to R$ on homotopy. (We will discuss good candidates for $R$ below.) 

In such a situation, it is reasonable to base-change from the category of $\sphere$-modules to the category of $R$-modules. In practice, this is a three-step process.
We consider the $R$-module spectra $R\otimes \Sigma^\infty X$ and $R\otimes \Sigma \cp{r}{n}$; attempt to compute stable homotopy classes of $R$-module maps, which we denote by
\[ \hmaps{R}{R\otimes \Sigma^\infty X}{R\otimes \Sigma \cp{r}{n}};\]
and try to show that certain classes in it are in the image of the induced Hurewicz map
\begin{equation}\label[diagram]{hurewicz-0}\hmapsp{\Sigma^\infty X}{ \Sigma \cp{r}{n}} \to \hmaps{R}{R\otimes \Sigma^\infty X}{R\otimes \Sigma \cp{r}{n}}.\end{equation}

This is the strategy carried out in \cite{HMY} for $R=\eop$, a higher real $K$-theory of height $p-1$ at a given prime $p$, and $X$ a complex projective space $\cp{}{n}$. The main results of that paper are to uniformly detect certain infinite families of $p$-power torsion in $\hmapsp{\cp{r}{n}}{ \Sigma \cp{r}{n}} $ by constructing $\eo$-module maps and proving that \Cref{hurewicz-0} is surjective in the range $n-r<2p^2-p-2$ (see \cite[Theorem D]{HMY}).
\begin{rmk} The same methods could be used to detect $p$-power torsion in stable homotopy groups of unitary groups; however, this does not lead to new results. We refer the reader to \Cref{warmup} for more discussion. \end{rmk} 

For $c\geq 2p^2-p-2$, the methods from \cite{HMY} do not apply verbatim: we do not have a general detection result to guarantee that maps of $\eop$-modules lift to maps of spectra. However, one can still try to detect non-zero maps of spectra by constructing specific, non-zero $\eop$-module maps and showing that they lift to spectra, thus identifying interesting stably trivial, non-trivial vector bundles (see \cite[Section 5.4]{HMY} for some ideas along these lines). This is the strategy that we take in this paper, both for $\eop$ and for other choices of $R$.

 Before actually carrying out this strategy, we survey the choices of $R$ that will be useful.

\subsubsection{Real $K$-theory}\label{KO}

Various flavors of $K$-theory are among the most ubiquitous geometrically constructed $E_\infty$-ring spectra in homotopy theory. Both real and complex $K$-theory can be constructed in many different ways; but, classically and perhaps most intuitively, they can be built concretely from unitary or orthogonal groups, and are essentially completely understood. We focus on the real $K$-theory $KO$. 

By real Bott periodicity, the homotopy groups of $KO$ are $8$-periodic, with
\[
\pi_kKO\cong
\begin{cases}
\mathbb Z & k\equiv 0,4\pmod 8,\\
\mathbb Z/2 & k\equiv 1,2\pmod 8,\\
0 & \text{otherwise}.
\end{cases}
\]
Adams showed that the $KO$-Hurewicz homomorphism
\[
\pi_*\sphere\longrightarrow \pi_*KO
\]
is surjective in degrees congruent to $1$ or $2$ modulo $8$ \cite{Adams}; see also \cite[Section 1]{BMQ}. In particular, for every $t\geq 0$, the element
\[
\alpha_{4t+1}\in\pi_{8t+1}\sphere
\]
maps to the nonzero element of $\pi_{8t+1}KO\cong\mathbb Z/2$.

\subsubsection{Topological modular forms}\label{tmf}

The spectrum $\op{tmf}$ of topological modular forms is a celebrated object in homotopy theory; it can be constructed as the connective cover of global sections of a sheaf of spectra on a moduli stack of elliptic curves. Its existence and geometric properties are quite remarkable, so we refer to \cite{DFHH} for an exposition and further sources on $\op{tmf}$. For our purposes, we will only use that $\tmf$ is an $E_\infty$-ring spectrum and that its $2$-local homotopy ring has a certain form. 

A summary of the homotopy of $\tmf$ is contained in \cite[Part 1, Chapter 13, Section 1]{DFHH}. 
The $2$-primary $\tmf$-Hurewicz image was computed by Behrens--Mahowald--Quigley \cite[Theorem 1.2]{BMQ}. One consequence is that, for every $t\geq0$, there are nonzero classes named
\[
\Delta^{8t}w\in\pi_{192t+45}\tmf_{(2)}
\qquad\text{and}\qquad
\Delta^{8t}w\bar{\kappa}^{4}\in\pi_{192t+125}\tmf_{(2)}
\]
that belong to the image of
\[
\pi_*\sphere\longrightarrow\pi_*\tmf_{(2)}.
\]

Thus, $\tmf_{(2)}$ detects infinite families of $2$-primary classes in the stable homotopy groups of spheres.

\subsubsection{Higher real $K$-theories}\label{EO}

The higher real $K$-theories of height $p - 1$ are constructed as follows. 

Let $\mathbb F=\mathbb F_{p^{p-1}}$ and let
$\Gamma$ be the Honda formal group of height $(p-1)$ over $\bb{F}$. The universal
deformation of $\Gamma$ is classified by the Lubin--Tate ring
\[R:= W(\bb{F}) \llbracket v_1, \ldots, v_{p-2} \rrbracket, \] 
where $W(\bb{F})$ is the ring of Witt vectors. There is an associated even
periodic cohomology theory $\e_{p\,\text{-}1}$, called the Morava $E$-theory of height $(p-1)$ or
a Lubin--Tate theory, with
\[\pi_*\e_{p-1} \cong W(\bb{F}) \llbracket v_1, \ldots, v_{p-2} \rrbracket [u^{\pm1}],\] 
where the $v_i$'s are of degree zero and $u$ is of degree $-2$. 

Recall that the Morava $E$-theory $\e_{p\,\text{-}1}$ comes equipped with an action of the (small) Morava
stabilizer group $\text{Aut}(\Gamma)$. 

There is a maximal finite subgroup $G\cong C_p\rtimes C_{(p-1)^2}$ of $\text{Aut}(\Gamma)$ containing an
element of order $p$, unique up to conjugacy. We define 
\[\eop= \e_{p\,\text{-}1}^{\text{h}G},\]
the homotopy fixed points of $\e_{p\,\text{-}1}$ with respect to $G$.
\begin{rmk} 
These spectra are called higher real $K$-theories because $\eo_1$ at the prime $2$ is the $K(1)$-localization of real $K$-theory.  

At $p=3$,  it is useful to note the group $G$above is contained in an order $24$ finite subgroup $G_{24}$ of the extended Morava stabilizer group; the additional $C_2$ is the action of $\operatorname{Gal}(\mathbb F_9/\mathbb F_3)$. The inclusion $G \subseteq G_{24}$ induces
a restriction map \[ L_{K(2)}TMF\simeq E_2^{hG_{24}}\longrightarrow E_2^{hG}=\eo_2. \]
\end{rmk}
\begin{rmk}More generally, homotopy fixed-point spectra $E_n^{hH}$, where
$H$ is a finite subgroup of the height-$n$ Morava stabilizer group,
are often called forms of higher real $K$-theory, and denoted $\eo_n$ when $H$ is understood. \end{rmk}

For odd $p$, the positive-filtration part of the homotopy fixed-point
spectral sequence computing $\pi_*\eop$ was calculated in unpublished work
of Hopkins and Miller. A published account of its $E_2$-page and differentials is given in \cite[Section 2.2, Theorem 2.1 and p.~499]{Nave}; see also
\cite[Section 5.2]{Chat} for a brief review. In particular, these calculations show that $\eop$ detects the first $p$-torsion elements in the stable homotopy groups of spheres: the element $\alpha_1\in\pi_{2p-3}\sphere$ has nonzero Hurewicz image, as does its $p$-fold Toda bracket $\beta_1\in\pi_{2p^2-2p-2}\sphere$. Hence, $\eop$ is a good candidate to detect vector bundles built up from $\alpha_1$-attaching maps.

At $p=3$, the Hurewicz image for $\eo_2$ is infinite. In \cite[Theorem 6.5]{BelShim}, Belmont and Shimomura prove that for each $t \geq 0$ 
there is an element $\theta_t \in \pi_{37+72t}\sphere$ with $\tmf$-Hurewicz image the nonzero element called
\[\beta_1 [\alpha_1\Delta]\Delta^{3t}\in\pi_{37+72t}\tmf.\]
For instance, $\theta_0$ is the class detected in the Novikov spectral sequence by $\alpha_1\beta_{3/3}$. Belmont--Shimomura also prove that these classes have nonzero image under the natural map $\pi_*\tmf \to \pi_*(L_{K(2)}TMF)$ \cite[Lemma 6.6]{BelShim}. Since $L_{K(2)}TMF \simeq \e_2^{hG_{24}}$ and the restriction $\e_2^{hG_{24}}\to \eo_2$ is split-injective on $3$-complete homotopy, this implies the image of the same classes is non-zero in $\eo_2$.

For $p\geq 5$, the Hurewicz image in $\pi_*\eo_{p-1}$ is not known to be infinite, and in fact is widely suspected to be finite. However, a finite family in the Hurewicz image exists by work of Ravenel \cite[Theorem 4.4.22 and the discussion following it]{Ravenel_green} and Nave \cite[p.~499]{Nave}. (We note that these arguments also applies when $p=3$, but that in fact we get a much larger Hurewicz image.) To describe these classes in detail, we follow \cite[pages 496--500]{Nave}, which computes the homotopy groups of $\eo_{p-1}$ via a homotopy fixed point spectral sequence.

 First, note that there is a map
\[\Ext_{\bp_*\bp}(\bp_*, \bp_*) \rightarrow \Ext_{\e_*^{\eo}\e}(\e_*, \e_*)\]
from the $\e_2$-page of the Novikov spectral sequence to that of the homotopy
fixed point spectral sequence. 

This map takes $\alpha_1$ to a class $\alpha$; $\beta_1$ to a class called $\beta$; and $\beta_{p/p}$ to a class $\Delta\beta$. For $1\leq j \leq p-2$, the class $\alpha\Delta\beta^j$ is a permanent cycle in the homotopy fixed point spectral sequence, detecting a class 
\[ [\alpha \Delta]\beta^j \in \pi_{d_j}\eo_{p-1},\] 
where
\begin{equation}\label{def:dj}d_j=2p(p-1)^2+2p-3+j(2p^2-2p-2).\end{equation}
The differential computations on page 499 of \cite{Nave} show that $ [\alpha \Delta]\beta^j $ is non-zero for $1\leq j \leq p-2$. Moreover, $[\alpha \Delta]\beta^j$ is the Hurewicz image of the element
$\theta'_j\in \pi_{d_j}\bb{S}$ detected by $\alpha_1\beta_{p/p}\beta_1^{j-1}$ in the Novikov spectral sequence.

\begin{rmk} For $p=3$, note that $\theta'_1=\theta_0$.\end{rmk} 

\begin{rmk}\label{divisibility}
It will later be useful to note that the Hurewicz image of $\theta'_j$ for $1\leq j \leq p-2$ is not divisible by $\alpha_1$. Nave’s calculation of the homotopy fixed point spectral sequence shows that, for $1\leq j\leq p-2$, the nonzero class $[\alpha\Delta]\beta^j$ is not divisible by $\alpha_1$ \cite[Theorem 2.1 and Proposition 2.3; see also the differential calculation on p.~499]{Nave}. Indeed, $\Delta\beta^j$ supports a nonzero \(d_{2p-1}\)-differential, and the description of the $E_\infty$-page rules out any other possible $\alpha_1$-divisor.
\end{rmk}

\section{Detection results for unstable homotopy of unitary groups}\label{sec:unitary-detection} 

We now explore the main results for homotopy groups of unitary groups. We begin in \Cref{warmup} with some preliminary calculations that illustrate the basic detection principles we will use, although they do not yield new resultson the homotopy groups of interest.

\subsection{Warm-up: \texorpdfstring{$p$}{p}-torsion in the homotopy groups
\texorpdfstring{$\pi_{2n}BU(r)$}{pi_2n BU(r)} for
\texorpdfstring{$n-r<\frac{2p^2-p-2}{2}$}{n-r < (2p²-p-2)/2}}
\label{warmup}
Fix an odd prime $p$. Throughout this subsection, let $\eo=\eop$ denote higher real $K$-theory of height $p-1$ at the prime $p$ (see \Cref{EO} for details on this cohomology theory).

We first consider which vector bundles on spheres can be detected in $\hmaps{\eop}{\cp{r}{n}}{\Sigma\cp{r}{n}}$ for $p$ not too small relative to the corank. We will find that, in this range, we do not get any new information. However, it is interesting to compare to classical answers.

By \cite[Propositions 5.17 and 6.4]{HMY}, we see that
\begin{prop}\label{prop:htpy_cpnr_eo} Let $n>r>1$ be given and let $c=n-r$. For $c<\frac{2p^2-p-2}{2}$,
\[ \eo_{-1} \Sigma^{-2n}\cp{r}{n} = \begin{cases}
\Z/p \text{ if } (p-1) \leq c< p(p-1) \text{ and } 
\lfloor c/(p-1)\rfloor -1\equiv [r]_p+[c]_{p-1}\pmod{p}
 \\ 0 \text{ otherwise.}\end{cases}\]

In the above, $[k]_m$ is the residue of $k \pmod{m}$ as an element in $\{0,\ldots, m-1\}.$
\end{prop}

The surjectivity result \cite[Corollary 5.13]{HMY} together with an $\eo$-module splitting of $\eo \otimes \Sigma \cp{r}{n}$ has consequences for
 the Hurewicz image \begin{equation}\label[empty]{hurewicz-sphere}\pi_{2n} \Sigma \cp{r}{n} \to (\eo)_{2n}\Sigma \cp{r}{n}.\end{equation}
More explicitly, for metastable $n,r$ such that
 $n-r< 2p^2-p-2$, the map \Cref{hurewicz-sphere} is surjective. 
Therefore, we might hope to obtain new results on the homotopy groups of $U(r)$ in this range from \Cref{prop:htpy_cpnr_eo}: by Hurewicz surjectivity, every nonzero class appearing in the proposition lifts to a $p$-power torsion class in the corresponding homotopy group of $U(r)$.
However, a stronger result was already proven by Matsunaga:
\begin{thm}[{\cite[Main Theorem]{Matsunaga64}}]
    \label{thm:matsunaga}

    Fix an arbitrary odd prime $p$. Let $n\geq 2$, and let $k$ be a positive
    integer such that:
    \begin{itemize} 
    \item $k \leq p (p-1)$;
    \item $n>k$; and
    \item $n+k \equiv 0 \pmod{p}.$
    \end{itemize}
    Then the $p$-primary  component of $\pi_{2n+2k-3}U(n)$ is $\Z/p^N$, where \[N=\min\left\{\left\lfloor\frac{k-1}{p-1}\right\rfloor,v_p(n+k)\right\},\] and $v_p(n+k)$ is the highest
    exponent of $p$ dividing $n+k$.
\end{thm}

Evidently, surjectivity of the Hurewicz image \Cref{hurewicz-sphere} combined with the $\eo$-homology computation above does not recover \Cref{thm:matsunaga}, since the classes in $\eo_{-1}\cp{r}{n}$ detect
    $p$-torsion in $\pi_{*}\cp{r}{n}$, but our methods do not address
    $p$-divisibility. However, once these classes are identified as permanent
    cycles in the Adams spectral sequence for $\pi_{*}\cp{r}{n}$, inspection of
    extensions might be used to produce the divisibility given above. This suggests an alternate proof of \Cref{thm:matsunaga}.

\subsection{Higher-degree \texorpdfstring{$p$}{p}-torsion in the homotopy groups of unitary groups}\label{subsec:higher_p_unitary}

In this section, we use ad hoc methods to detect new $p$-torsion in $\pi_{2n}BU(r)$. These results use classes in the
Hurewicz image for $KO$, $\tmf_{(2)}$, and $\eop$. We will generalize these arguments to produce non-trivial, stably trivial vector bundles on complex projective spaces in
\Cref{sec:detect_beyond}.

We first prove an elementary lemma to set the stage.

\begin{lemma}\label{lem:split_construction_sphere} 
Let $R$ be a commutative ring spectrum and let $\gamma_1$ be the universal
bundle on $\cp{}{\infty}$. Let $n$ and $d$ be positive integers, with $d$ odd. Suppose that $n\gamma_1$ is $R$-orientable, and
$\theta \in \pi_{d}\sphere$ has nonzero Hurewicz image in $\pi_{d}R$.
Then for $i$ such that $ni\geq \frac{d+1}{2}$ the vector bundle associated to the composite
\begin{equation}\label[diagram]{composite1}
V_\theta\: \sphere^{2ni+d+1} 
\xrightarrow{\theta}\sphere^{2ni+1} \xrightarrow{b} 
\Sigma \cp{ni}{ni+\frac{d+1}{2}}\end{equation} 
is nontrivial, where $b$ is the map from the bottom cell into the finite spectrum.
\end{lemma}

\begin{proof} First we show that the bottom cell of $\cp{ni}{ni+\frac{d+1}{2}}$ 
splits after tensoring with $R$. By the hypothesis that $n\gamma_1$ is $R$-orientable,

\[ R\otimes \cp{ni}{ni+\frac{d+1}{2}}\simeq R \otimes \Sigma^{2ni} \cp{0}{\frac{d+1}{2}},\] so the bottom cell splits. 
The fact that $\theta$ has a nonzero Hurewicz image implies that the composite 
$R \otimes V_\theta$ is not null.
\end{proof}

\begin{rmk}\label{rmk:orient} The condition that $n\gamma_1$ is $R$-orientable in \Cref{lem:split_construction_sphere} 
can be described in the language of the orientation order introduced by Bhattacharya--Chatham \cite{BhatChat}. The complex orientation order $\Theta(R)$ of a ring spectrum $R$, defined as the smallest $n$ so that $n\gamma_1$ is $R$-orientable, measures how far $R$ is from being complex-orientable. 
The orientation order of $\eo_{\Gamma}$ has been studied extensively by Bhattacharya--Chatham:
$\Theta(\eo_{\Gamma})$ divides $p^{p^k-1}$ when the height of the formal group 
$\Gamma$ is $n = (p-1)k$, and conjecturally $\Theta(\eo_{\Gamma}) = p^k$  \cite[Main Theorem 1.6 and Conjecture 1.13]{BhatChat}. 
When $k=1$, the first author proves that $\Theta(\eop) = p$ \cite[Corollary 1.6]{Chat}.
\end{rmk}

\begin{rmk} The previous lemma is set up to identify situations where non-trivial, stably trivial vector bundles can be detected in $R$-theory essentially for free, by arranging that the bottom cell in the target stunted projective spectrum splits after tensoring with $R$, and that the defining map of the vector bundle factors through this summand after tensoring with $R$. One can, of course, try to construct more complicated bundles without these hypotheses, but our goal is to illustrate that infinite families of interesting bundles can be detected even without doing new computations about the $R$-theory of $\cp{r}{n}$.\end{rmk}

\subsubsection{$2$-torsion in homotopy of unitary groups via $KO$}\label{sphere-KO}
 We first apply \Cref{lem:split_construction_sphere} with $R=KO$ and appropriately chosen classes in the Hurewicz image. We refer the reader to \Cref{KO} for relevant background on the homotopy and Hurewicz image of real $K$-theory. Consider the $2$-torsion elements $\alpha_{4t+1} \in
\pi_{8t+1}\sphere$. Since $\alpha_{4t+1}$ is detected in $KO$ (see \cite{Adams} or \cite[Section 1]{BMQ}) and
$2 \gamma_1$ is 
$KO$-orientable \cite[Lemma 2.18]{BhatChat}, we apply \Cref{lem:split_construction_sphere} to obtain:

\begin{prop}\label{cor:KO_htpy_BU} For each $t\geq 0$ and each $i \geq 2t+1$, there is a nontrivial rank $2i$ bundle on $S^{4(2t+1+i)-2}$
 that gives nonzero $2$-torsion in $\pi_{4(2t+1+i)-2} BU(2i).$
\end{prop}

\subsubsection{$2$-torsion in homotopy of unitary groups via $\tmf$ localized at $2$}\label{sphere-tmf}
Continuing $2$-primary calculations, consider $R=\tmf_{(2)}$, as discussed in \Cref{tmf}. By Bauer
\cite[Lemma 2.1]{Bauer03}, $8\gamma_1$ is $\tmf_{(2)}$-orientable. By
Behrens--Mahowald--Quigley \cite{BMQ}, there are classes 
$\Delta^{8t}\cdot w \in
\pi_{192t+45}\tmf_{(2)}$ and $\Delta^{8t}\cdot w \bar{\kappa}^4 \in
\pi_{192t+125}\tmf_{(2)}$ that are in the Hurewicz image. By \Cref{lem:split_construction_sphere}:

\begin{cor}\label{cor:tmf2_htpy_BU} For each $t \geq 0$,
$i\geq 12t+3$, and $j \geq 12t+8$, let $m(i)= 8(12t+3+i)-1$ and $n(j)=8(12t+8+j)-1$. There are nontrivial rank 
$8i$ bundles on $S^{2m(i)}$ and rank 
$8j$ bundles on $S^{2n(j)}$ 
that give nonzero $2$-torsion in $\pi_{2m(i)}BU(8i)$ and $\pi_{2n(j)}BU(8j)$, respectively.
\end{cor}

\subsubsection{$p$-torsion in homotopy of unitary groups via $\eop$}\label{sphere-eo3}
Now we move on to $p$-torsion for $p$ an odd prime. We consider the theories $\eop$ at the implicit prime $p$, and refer the reader to \Cref{EO} for more background, including discussion of the specific homotopy classes that we will use.  We also repeatedly use that $p\gamma_1$ is $\eop$-orientable (see \Cref{rmk:orient} and \cite[Corollary 1.6]{Chat}).

We get slightly stronger results for $p=3$ than for $p\geq 5$. For $p=3$, consider $\eo_2$: as discussed in \Cref{EO}, \cite[Theorem 6.5]{BelShim} together with a restriction argument shows that, for each $t \geq 0,$ 
there is an element $\theta_t \in \pi_{37+72t}\sphere$ with non-zero Hurewicz image in $\pi_{37+72t}\eo_2$.
So, by \Cref{lem:split_construction_sphere} and Chatham's $3$-primary $\eo_2$-orientability result \cite[Corollary 1.6]{Chat}:

\begin{cor}\label{cor:bel-shim2}
    For each $t\geq 0$ and each $l \geq 12t+7$, there is a nontrivial rank $3l$
    bundle on $S^{2(3l+19+36t)}$ that gives nonzero $3$-torsion in
    $\pi_{2\left(3l+19+36t\right)} BU(3l)$.
\end{cor}

We now consider an arbitrary odd prime $p\geq 5$. As discussed in \Cref{EO}, for each $j$ such that $1\leq j\leq p-2$, there is an element
$\theta'_j\in \pi_{d_j}\bb{S}$ with nonzero Hurewicz image in $\pi_{d_j}\eop$, where
\[d_j=2p(p-1)^2+2p-3+j(2p^2-2p-2).\]
By \Cref{lem:split_construction_sphere}:

\begin{cor}\label{example:general_odd_prime2}
    Let $p\geq 5$ be an odd prime. Let $d_j$ be as in \Cref{def:dj}. For each $1\leq j \leq p-2$ and each $l\geq
    \frac{d_j+1}{2p}$, there is a nontrivial rank $lp$ bundle over
    $S^{2lp+1+d_j}$ that gives nonzero $p$-torsion in $\pi_{2lp+1+d_j}BU(lp)$.
\end{cor}



\section{Detection results for metastable, stably trivial vector bundles on complex projective spaces beyond corank \texorpdfstring{$2p^2-p-3$}{2p²−p−3}}\label{sec:detect_beyond}

 In \cite{HMY}, we detect non-trivial, stably trivial rank $r$ complex topological vector bundles on $\cp{}{n}$, for $n$ and $r$ metastable, by building certain classes in the higher real $K$-theory of a certain finite spectrum. For $p$ an odd prime, in \cite[Theorem 5.19]{HMY} we compute this group completely for $n-r<\frac{2p^2-p-2}{2}$. The remark following that theorem identifies an explicit direct summand throughout the larger range $n-r<2p^2-p-2$, while \cite[Theorem D]{HMY} shows that the corresponding Hurewicz map is surjective throughout this larger range.
These results are only useful for a range of coranks relative to the prime: at most, the regime is $n-r<2p^2-p-2,$ depending on the result being invoked. 

In this section, we show that the arguments in \Cref{sec:unitary-detection} above for spheres can be adapted to the case of projective spaces in order to detect nontrivial, high-corank vector bundles on complex projective spaces --- both at the prime $2$ and at odd primes. In particular, we are able to construct many bundles beyond corank $n-r> 2p^2-p-3$, which is the upper bound for odd-primary constructions in \cite{HMY}. The vector bundles detected in these sections are all pullbacks of vector bundles on spheres detected in \Cref{subsec:higher_p_unitary}.

 In \Cref{subsec:splitcells}, we directly imitate the $2$-primary arguments from \Cref{sec:unitary-detection} to detect nontrivial bundles that are pullbacks of vector bundles on spheres. The key observation here is that, for appropriately chosen $R$ and certain values of $n$ and $r$, the top cell in $\cp{r}{n}$ splits after tensoring with $R$.

\Cref{subsec:detect_p3} and \Cref{example:general_odd_prime} contain slightly more involved constructions to detect $p$-power torsion for $p$ an odd prime. These run parallel to \Cref{cor:bel-shim2} and \Cref{example:general_odd_prime2}, but require some extra work.

\subsection{2-primary detection of high-corank vector bundles on complex projective spaces}\label{subsec:splitcells}
The common tool that we will use in this section is the following lemma, which is an analogue of \Cref{lem:split_construction_sphere}:

\begin{lemma}\label{lem:split_construction} 
Let $R$ be a commutative ring spectrum and let $\gamma_1$ be the canonical
bundle on $\cp{}{\infty}$. For $n,k$ positive integers, suppose that $n\gamma_1$ is $R$-orientable and
$\theta \in \pi_{2nk-3}\sphere$ has nonzero Hurewicz image in $\pi_{2nk-3}R$.
Then for $i\geq k$ the vector bundle associated to the composite
\begin{equation}\label[diagram]{composite2}V_\theta\: \cp{ni}{n(k+i)-1} \xrightarrow{a} 
\sphere^{2n(k+i)-2} 
\xrightarrow{\theta}\sphere^{2ni+1} \xrightarrow{b} 
\Sigma \cp{ni}{n(k+i)-1}\end{equation} 
is nontrivial, where the maps $a$ and $b$ are the maps from the source to its
top cell and the bottom cell into the target, respectively.
\end{lemma}

\begin{proof} First we show that both the top and bottom cells of $\cp{ni}{n(i+k)-1}$ 
split after tensoring with $R$. By the hypothesis that $n\gamma_1$ is $R$-orientable,

\[ R\otimes \cp{ni}{n(i+k)-1}\simeq R \otimes \Sigma^{2ni} \cp{0}{nk-1},\] so the bottom cell splits. 
To show that the top cell splits after smashing with $R$, we show that the bottom cell of 
$R\otimes D \cp{ni}{n(i+k)-1}$ splits. Let $T$ denote the tangent bundle on $\cp{}{nk-1}$. 
Since $T \oplus \underline{\mathbb C} \simeq \gamma_1^{\oplus nk}$, 

\begin{align*} R\otimes D \cp{0}{nk-1} 
& \simeq R\otimes (\cp{}{nk-1})^{-T}\\
&\simeq  R\otimes \Sigma^2 (\cp{}{nk-1})^{-(nk)\gamma_1}\\
&\simeq R\otimes \Sigma^{2-2nk}\cp{0}{nk-1}.\end{align*}
Hence the maps $a$ and $b$ in \Cref{composite2} are split after tensoring with 
$R$. The fact that $\theta$ has a nonzero Hurewicz image implies that the composite 
$R \otimes V_\theta$ is not null.
\end{proof}

We will apply \Cref{lem:split_construction} with both $R=KO$ and $R=\tmf_{(2)}$; we refer the reader to \Cref{KO} and \Cref{tmf} for background discussion, and in particular for a description of relevant classes in the homotopy groups and Hurewicz images of these spectra.

\subsubsection{Vector bundles on projective spaces via $KO$}\label{proj-KO} Again consider the $2$-torsion elements $\alpha_{4t+1} \in \pi_{8t+1}\sphere$. Since $\alpha_{4t+1}$ is detected in $KO$ \cite{Adams} and $2 \gamma_1$ is 
$KO$-orientable \cite[Lemma 2.18]{BhatChat}, we apply \Cref{lem:split_construction} and obtain the following result.
\begin{cor}\label{cor:KO_ah} For each $t\geq 0$ and $i \geq 2t+1$, there is a nontrivial rank $2i$ bundle on $\cp{}{2i+4t+1}$
 that gives nonzero $2$-torsion in 
\[\Vect_{2i}^0(\cp{}{2i+4t+1})\cong \hmapsp{\cp{2i}{2i+4t+1}}{\Sigma\cp{2i}{2i+4t+1} }.\]
\end{cor}

\subsubsection{Vector bundles on projective spaces via $\tmf_{(2)}$}\label{proj-tmf}

By Bauer \cite[Lemma 2.1]{Bauer03}, the bundle $8\gamma_1$ is 
$\tmf_{(2)}$-orientable. By Behrens--Mahowald--Quigley \cite{BMQ}, there are nonzero classes $\Delta^{8t}\cdot w \in \pi_{192t+45}\tmf_{(2)}$ and $\Delta^{8t}\cdot w \bar{\kappa}^4 \in \pi_{192t+125}\tmf_{(2)}$
that are in the Hurewicz image. Appealing to \Cref{lem:split_construction}:
\begin{cor}\label{cor:tmf_ah} For each $t \geq 0$, $i\geq 12t+3$, and $j \geq 12t+8$, 
there are nontrivial rank $8i$ bundles on $\cp{}{8(12t+3+i)-1}$ and rank $8j$ bundles on $\cp{}{8(12t+8+j)-1}$. These bundles give nonzero $2$-torsion classes in 
\[\Vect_{8i}^{0}(\cp{}{8(12t+3+i)-1})
\simeq \hmapsp{ \cp{8i}{8(12t+3+i)-1}}{\Sigma\cp{8i}{8(12t+3+i)-1} }\] and in
\[\Vect_{8j}^{0}(\cp{}{8(12t+8+j)-1})\simeq \hmapsp{ \cp{8j}{8(12t+8+j)-1}}{\Sigma\cp{8j}{8(12t+8+j)-1} },\] respectively.
\end{cor}

\subsection{ Odd-primary detection of vector bundles on complex projective spaces  }\label{subsec:detect_p3}
In \Cref{subsec:splitcells}, we used \Cref{lem:split_construction} 
 to produce nontrivial vector bundles from elements in the Hurewicz 
image of $R=KO$ or $R=\tmf_{(2)}$. 
The key observation was that $R$-orientability 
for multiples of the universal bundle on $\cp{}{\infty}$ leads to examples of 
stunted projective spaces $\cp{r}{n}$ where the top and bottom cells split after 
tensoring with $R$. This allowed us to show that a composite 
$\cp{r}{n} \to \sphere^{2n} \to \sphere^{2r+1} \to \Sigma \cp{r}{n}$ 
is nontrivial after tensoring with $R$ by an easy splitting argument. 

Requiring the top and bottom cells of $R\otimes \cp{r}{n}$ to split rigidly determines the 
degree of an element in $\pi_*\sphere$ that can be used in \Cref{lem:split_construction}.
 Considering more general $R$-split 
summands that are still simpler than $\cp{r}{n}$ itself, 
we extend our methods. However, the proof requires slightly more work.

\subsubsection{Vector bundles on projective spaces via $\eo_2$}\label{proj_eo3}

Recall from \Cref{EO} and \cite[Theorem 6.5]{BelShim} that, for each $t \geq 0$,
there is an element $\theta_t \in \pi_{37+72t}\sphere$ whose image in $\pi_{37+72t}\eo_2$ is the nonzero class
$\beta_1 [\alpha_1\Delta]\Delta^{3t}$. By \cite[Corollary 1.6]{Chat}, the bundle $3 \gamma_1$ is $\eo_2$-orientable. 

\begin{prop}\label{example:bel-shim} 
For each $t\geq 0$ and $l$ such that $3l \geq 19+36t$, there is nonzero $3$-torsion in the set of rank $3l$ bundles on $\cp{}{3l+19+36t}$, associated to the element $\theta_t$.\end{prop}
\begin{proof}
All spectra that follow are $3$-completed, so in particular stunted projective spectra admit $3$-complete splittings which we will utilize extensively. Let $n=3l+19+36t$ and $r=3l$.
 To prove the proposition, we first define a map of spectra $V_{\theta_t}\:\cp{r}{n} \to \Sigma \cp{r}{n}$ that we will detect with $\eo_2$. Consider the following diagram:
\begin{equation}
\begin{tikzcd}
\cp{r}{n}  \ar[rr,"V_{\theta_t}"]\ar[d,"a"]  && \Sigma \cp{r}{n}\\
\Sigma^{2n-4}C(\alpha_1) \ar[rr,"\tilde\theta_t"]\ar[dr,"b"]& & \sphere^{2r+1} \ar[u,"c"]\\
 & \sphere^{2n} \ar[ur,"\theta_t"]
\end{tikzcd}
\end{equation}
The maps in the above diagram are as follows:
\begin{itemize}
\item $a\: \cp{r}{n} \to \Sigma^{2n-4}C(\alpha_1)$ is the composite of the map $p$ in the cofiber sequence 
\[\cp{r}{n-3}\to \cp{r}{n} \xrightarrow{p} \cp{n-2}{n} \simeq \Sigma^{2n-4}C(\alpha_1)\oplus \sphere^{2n-2}\] 
with the projection $\cp{n-2}{n} \to \Sigma^{2n-4} C(\alpha_1)$; 
\item $b$ fits into the cofiber sequence $\sphere^{2n-4} \to \Sigma^{2n-4}C(\alpha_1) \xrightarrow{b} \sphere^{2n}$;
\item $c$ is the inclusion of the bottom cell in $\Sigma \cp{r}{n}$; 
\item $\tilde \theta_t = \theta_t \circ b$;  and
\item $V_{\theta_t}$ makes the diagram commute.
\end{itemize}
First, we show that the maps $\eo_2 \otimes a$ and $\eo_2\otimes c$ split in $\eo_2$-modules. 
Note that \[\eo_2 \otimes \cp{3l}{n} \simeq \eo_2 \otimes \Sigma^{6l}\cp{0}{n-3l}\] by \cite[Corollary 1.6]{Chat}. It is immediate that 
$\eo_2 \otimes c$ splits. Using \cite[Theorem 5.6]{Chat}, $\eo_2\otimes \cp{0}{n-3l}$ splits as a sum of spectra 
$\eo_2\otimes\Sigma^{2s}X_i$ for $1 \leq i \leq 3$ (see \Cref{subsec:conventions} for a description of the spectra $X_i$).

Since $n \equiv 1\pmod 3$, the top summand is 
$\eo_2\otimes \Sigma^{2n-4}C(\alpha_1)$ \cite[Proposition 6.4]{HMY}.
To conclude that $V_{\theta_t}$ is nontrivial it suffices to show that 
\[\eo_2\otimes \tilde \theta_t  \: \eo_2\otimes \Sigma^{2n-4}C(\alpha_1)\to \Sigma^{2r+1}\eo_2\]
is nonzero.  By the free--forgetful adjunction, it suffices to show that the composite
\[\Sigma^{2n-4}C(\alpha_1)\xrightarrow{\widetilde\theta_t} \sphere^{2r+1}\longrightarrow\Sigma^{2r+1}\eo_2
\]
is nonzero. By the cofiber sequence defining $C(\alpha_1)$, this composite
is zero if and only if the $\eo_2$-Hurewicz image of $\theta_t$ is divisible
by $\alpha_1$, which it is not: the fact that the Hurewicz image in $\e_2^{hG_{24}}$ is not $\alpha_1$-divisible follows from \cite[Theorem 6.1, Figure 5, and the proof of Lemma 6.6]{BelShim}. However, we need a bit more. The conclusion for the image in $\pi_*\eo_2$ follows from \Cref{divisibility}, or directly from the homotopy fixed-point spectral-sequence calculation
of \cite[Section~2.2, especially Theorem~2.1 and pp.~498--499]{Nave}. \end{proof}

\subsubsection{Vector bundles on projective spaces via $\eop$}\label{proj_eop}
\label{subsec:detect_higher_p}
We again refer to \Cref{EO} for background on higher real $K$-theories and discussion of homotopy and Hurewicz images.

Recall from \Cref{example:general_odd_prime2} that there is a
finite family in the Hurewicz image: for
each $j$, $1\leq j\leq p-2$, there is an element $\theta'_j\in \pi_{d_j}\bb{S}$
whose image in $\pi_{d_j}\eop$ is the nonzero class $[\alpha
\Delta]\beta^j$. Here the degree $d_j$ equals $2p(p-1)^2+2p-3+j(2p^2-2p-2)$. 

We might hope to use the odd-primary non-trivial vector bundles on spheres from \Cref{example:general_odd_prime2} to construct families of nontrivial, stably trivial $p$-power torsion vector bundles on appropriate-dimensional complex projective spaces for $1\leq j \leq p-2$. Unfortunately, the numerology in our argument only works out nicely when $j=1$. (Note, however, that we do not actually prove the vector bundles on projective spaces obtained from $\theta'_j$ for $2\leq j \leq p-2$ are trivial!)

\begin{prop}\label{example:general_odd_prime}
    Let $p\geq 5$ be an odd prime. 
Let $d=2p(p-1)^2+2p-3+(2p^2-2p-2)$.
 For each $l$ so that $lp\geq (d+1)/2$,
    there is nonzero $p$-torsion in the set of rank $lp$ bundles over
    $\cp{}{lp+(d+1)/2}$, associated to the element $\theta_1'\in
    \pi_{d}\bb{S}$.
\end{prop}

\begin{proof}
All spectra that follow are $p$-completed, so in particular stunted projective spectra admit $p$-complete splittings which we will use implicitly. 
    Let $r = lp$ and let $n = r+(d+1)/2$. Since $r\equiv 0 \pmod p$, by \cite{Chat} we have
    \[\eop\otimes\cp{r}{n}= \eop\otimes\cp{lp}{n} \simeq \eop\otimes\Sigma^{2lp}\cp{0}{n-lp}.\]
    Note that $\cp{0}{n-lp}$ splits as a sum of 
suspensions of the spectra $X_i$ after tensoring with $\eop$ \cite[Theorem 5.6]{Chat}. (See \Cref{subsec:conventions} for a description of the spectra $X_i$.)
 
   We now come to the point where we use the fact that we are working with $\theta'_1$, not a more general $\theta'_j$. Since $n\equiv (d+1)/2 \equiv p-2 \pmod p$, the top summand is 
$\eop\otimes \Sigma^{2n-2p+2}C(\alpha_1)$ \cite[Proposition 6.4]{HMY}. Projection onto this top $\eo$-split summand is in fact induced by a map of $p$-complete spectra tensored with $\eo$: to see this, note that projection onto an Adams summand followed by the quotient of a lower skeleton achieves a map \[a\:\cp{r}{n}\to \Sigma^{2n-2p+2} C(\alpha_1)\] whose base change to $\eo$-modules is the desired projection.

Consider the following diagram:
    \begin{equation}\label[diagram]{diagm:const_theta}
\begin{tikzcd}
\cp{r}{n}  \ar[rr,"V_{\theta_1'}"]\ar[d,"a"]  && \Sigma \cp{r}{n}\\
\Sigma^{2n-2p+2}C(\alpha_1) \ar[rr,"\tilde\theta'_1"]\ar[dr,"b"]& & \sphere^{2r+1} \ar[u,"c"]\\
 & \sphere^{2n} \ar[ur,"\theta'_1"]
\end{tikzcd}
\end{equation}
The maps in \Cref{diagm:const_theta} are as follows:
\begin{itemize}
\item $b$ fits into the cofiber sequence $\sphere^{2n-2p+2} \to \Sigma^{2n-2p+2}C(\alpha_1) \xrightarrow{b} \sphere^{2n}$;
\item $c$ is the inclusion of the bottom cell in $\Sigma \cp{r}{n}$; 
\item $\tilde \theta'_1 = \theta_1' \circ b$;  and
\item $V_{\theta'_1}$ makes the diagram commute.
\end{itemize}
The maps $a$ and $c$ become split after tensoring with $\eop$.

By the free--forgetful adjunction, it therefore suffices to show that the
composite
\[
\Sigma^{2n-2p+2}C(\alpha_1)
\xrightarrow{\widetilde\theta'_1}
\sphere^{2r+1}
\longrightarrow
\Sigma^{2r+1}\eop
\]
is nonzero. This composite
is zero if and only if the $\eop$-Hurewicz image of $\theta'_1$ is divisible
by $\alpha_1$. This is not the case by \Cref{divisibility}.
\end{proof}


\bibliographystyle{abbrv}
\bibliography{htpy-unitary}

\end{document}